\documentclass[reqno]{amsart}

\usepackage[a4paper, margin=1.4in]{geometry}

\usepackage{mathtools,amsthm,amssymb,amsfonts,amsmath}
\usepackage{enumitem}
\usepackage{graphicx,txfonts}
\usepackage{hyperref}
\usepackage{verbatim}
\usepackage{csquotes}

\usepackage{color}
\usepackage{tikz-cd}
\usetikzlibrary{cd}
\usepackage{setspace}
\usepackage{todonotes}
\usepackage[all]{xy}
\usepackage{quiver}

\usepackage{mathrsfs}
\usepackage{hyperref}

\usepackage[capitalise]{cleveref}
\usepackage{textgreek}
\usepackage{dutchcal}
\usepackage{pifont}

\theoremstyle{plain}
\newtheorem{theorem}{Theorem}[section]
\newtheorem*{theorem*}{The Contraction Principle}
\newtheorem{corollary}[theorem]{Corollary}
\newtheorem{proposition}[theorem]{Proposition}

\theoremstyle{definition}
\newtheorem*{definition*}{Definition}
\newtheorem{definition}[theorem]{Definition}

\newtheorem{example}[theorem]{Example}

\newtheorem{remark}[theorem]{Remark}
\newtheorem*{notation*}{Notation}

 \theoremstyle{definition}
 
 \theoremstyle{remark}

 \numberwithin{equation}{section}

\newcommand{\Coz}{\mathsf{Coz}}
\newcommand{\bl}{\mathsf{BL}} 
\newcommand{\dm}{\mathsf{DM}} 

\def\s{\sigma}

\newcommand{\fB}{{\mathfrak B}}
\newcommand{\fH}{{\mathfrak H}}
\newcommand{\fO}{{\mathfrak O}}
\newcommand{\pR}{{\mathfrak{pR}}}
\newcommand{\fp}{{\mathfrak{p}}}

\newcommand{\meet}{\wedge}
\newcommand{\join}{\vee}

\newcommand{\RR}{\mathbb{R}} 
\newcommand{\Ss}{\mathsf{S}}

\newcommand{\QQ}{\mathbb{Q}}
\newcommand{\PP}{\mathbb{P}} 
\newcommand{\MM}{\mathbb{M}} 
\newcommand{\Ml}{\mathfrak{O}\mathbb{M}} 
\newcommand{\NN}{\mathbb{N}}

\newcommand{\close}{\mathfrak{c}} 
\newcommand{\open}{\mathfrak{o}} 

\definecolor{myRed}{RGB}{214,92,92}
\definecolor{myGreen}{RGB}{53,154,25}
\definecolor{myBlue}{RGB}{92,92,214}

\def\setof#1#2{\{#1 \, |\, #2\} }  
\def\set#1{\{#1 \} } 
\newcommand{\Int}{\operatorname{int}}

\title[Perfect regularity in spaces and locales]{
An overlooked weakening of perfect normality:\\ Perfect regularity in spaces and locales}

\author{Ana Bel\'en Avilez}
\address{ITAM, Instituto Tecnológico Autónomo de México}
\email{anabelen.avilez@itam.mx}
\author{Guram Bezhanishvili}
\address{New Mexico State University}
\email{guram@nmsu.edu}
\author{Joanne Walters-Wayland}
\address{CECAT, Chapman University}
\email{joanne@waylands.com}

\begin{document}

\begin{abstract}
    We introduce the notion of  perfect regularity as an appropriate weakening of perfect normality, both for spaces and locales. Various characterizations are given, using Dedekind-MacNeille completions, injective hulls, and sublocales. We place the new class of perfectly regular frames among various well-studied classes of frames. We also introduce the construction of perfect regularization of a completely regular frame,  compare it to Isbell's well-known booleanization construction, and argue that it is at least as important as the latter.
\end{abstract}
\maketitle
\tableofcontents
\section{Introduction and motivation}

Normality is one of the most significant higher separation axioms in topology. Perfect normality further strengthens this by requiring each closed set to be a G$_\delta$-set (recall that a \emph{$G_\delta$-set} is a countable intersection of open sets). By a well-known result of Vedenissoff (see, e.g., \cite[Thm.~1.5.19]{engelkingGeneralToplogy1989}), $X$ is perfectly normal iff each closed set of $X$ is a zero-set, yielding a convenient pointfree description of perfect normality: the $\sigma$-frame of cozeros coincides with the frame of opens (see, e.g., \cite[Cor.~VIII.1.2.1]{picadoSeparationPointfreeTopology2021}).

It is natural to weaken perfect normality in the same vein that regularity weakens normality. An obvious candidate is to ask singletons to be G$_\delta$. However, in general, this is not equivalent to singletons being zero-sets, as was demonstrated recently by K.P.~Hart \cite{hartInfiniteLibrary2025}. This makes the analogy with perfect normality less obvious. This discrepancy disappears when the spaces are completely regular. We thus arrive at the following:

\begin{definition}
    A completely regular space is {\em perfectly regular} provided each singleton is G$_\delta$. 
\end{definition}

As we will see, this is equivalent to each singleton being a zero-set (\cref{PR for spaces}).
The condition that each singleton is a G$_\delta$-set was first studied by Chittenden \cite{chittendenGeneralTopologyRelation1929}.  Anderson \cite{andersonLatticeCharacterizationCompletely1955} termed these $G_\delta$-spaces\footnote{Note that in \cite[p.~162]{steenCounterexamplesTopology1978} this name is reserved for the stronger property that each closed set is G$_\delta$.}, while Aull \cite{aullBaseAxiomsTopology1968} referred to them as $E_0$-spaces.
A quintessential example of a perfectly normal space is the real line (with the usual topology), whereas the Michael line plays that role for perfectly regular spaces (see \cref{Michael is PR}). 

Our aim is to study perfect regularity from the perspective of pointfree topology. As we will see, this closely relates to Dedekind-MacNeille completions, a mainstream topic in lattice theory (see, e.g., \cite{balbesDistributiveLattices1974}), as well as to the important construction of injective hulls of semilattices by Bruns and Lakser \cite{brunsInjectiveHullsSemilattices1970} (see also \cite{hornCategorySemilattices1971}).

We provide various characterizations of perfect regularity, from which it follows that a space $X$ is perfectly regular iff the frame of opens $\fO X$ is isomorphic to the Dedekind-MacNeille completion of the $\sigma$-frame of cozeros $\Coz X$, which in turn is equivalent to $\fO X$ being isomorphic to the injective hull of $\Coz X$. We also place this new class of perfectly regular frames among various well-studied classes of frames in pointfree topology. In particular, \cref{PN=cozero+PR} shows that a frame is perfectly normal precisely when it is both perfectly regular and a cozero frame (the cozeros themselves form a frame). 

This approach gives rise to the  construction of the perfect regularization of a completely regular frame $L$ --- the Bruns-Lakser sublocale of $L$ generated by $\Coz L$ --- which we study in detail in \cref{sec: Perfect regularization}. Some results about this remarkable sublocale may be found in \cite{ballLindelofTightnessDedekindMacNeille2017}\footnote{In \cite{ballLindelofTightnessDedekindMacNeille2017} it is referred to as $sl(Coz L)$.}, \cite{bezhanishviliSemilatticeBaseHierarchy2024}\footnote{In \cite{bezhanishviliSemilatticeBaseHierarchy2024}, it is denoted by $D_\infty \Coz L$.}, and \cite{bezhanishviliDegreesJoindistributivityBruns2025}. We argue that it plays at least as prominent a role as Isbell's booleanization. Among other things, we give necessary and sufficient conditions determining when the two coincide (see \cref{char of almost P}), from which it follows that they coincide precisely when the frame is an almost P-frame (aka $\omega_1$-hollow).
Our findings enrich the understanding of perfect regularity in both frames and topological spaces, as well as provide new avenues for research, which are outlined in the concluding section.

\section{Perfect regularity}

In this section we introduce our key notion of a perfectly regular frame. To motivate the definition, we first give several equivalent conditions for a space to be perfectly regular. Throughout, we assume familiarity with standard topological concepts as presented in  \cite{engelkingGeneralToplogy1989}, as well as the basics of pointfree topology as outlined in \cite{picadoFramesLocalesTopology2012}.

\begin{theorem}\label{PR for spaces}
    For a completely regular space $X$, the following are equivalent:
    \begingroup
    \begin{enumerate}[label=\textup{(\arabic*)}]
    \item\label{item:PRsp-1} $X$ is a perfectly regular space.
    \item\label{item:PRsp-2} $\set{x}$ is a zero-set for each $x\in X$.
    \item\label{item:PRsp-3} Each closed set is the union of zero-sets contained in it.
    \item\label{item:PRsp-4} The union of all the zero-sets contained in a closed set is dense in the closed set. 
    \item\label{item:PRsp-5} Each open set is the interior of the intersection of the cozero-sets containing it.
    \end{enumerate}
    \endgroup
\end{theorem}
\begin{proof}
\ref{item:PRsp-1}$\implies$\ref{item:PRsp-2}: Let $x\in X$. Since $X$ is perfectly regular, $\set{x} = \bigcap_{n=1}^\infty U_n$ for some $U_n\in\fO X$. Let $F_n = X \setminus U_n$. Then each $F_n$ is closed and $x \notin F_n$. Since $X$ is completely regular, there is a continuous $f_n : X \to [0,1]$ such that $f_n(x) = 0$ and $F_n \subseteq f_n^{-1}(1)$. Therefore, $f_n^{-1}[0,1) \subseteq U_n$. Using a standard technique, let 
\[f = \displaystyle \sum_{n=1}^\infty \frac{1}{2^n}f_n.\]
Since this series is uniformly convergent, $f :X \to [0,1]$ is  continuous. Moreover, $\set{x} = f^{-1}(0)$: if $y \ne x$ then $y \notin U_n$ for some $n \ge 1$, so $f_n(y) = 1$, and hence $f(y)>0$. Thus, $\set{x}$ is a zero-set of $X$.
    
    \ref{item:PRsp-2}$\implies$\ref{item:PRsp-3} and \ref{item:PRsp-3}$\implies$\ref{item:PRsp-4} are obvious.

    \ref{item:PRsp-4}$\implies$\ref{item:PRsp-2}: Let $x \in X$ and consider the set $\set{x}$, which is closed (since $X$ is $T_1$). By \ref{item:PRsp-4}, the union of all the zero-sets contained in $\set{x}$ is dense in $\set{x}$. Thus, the union is nonempty, hence equal to $\set{x}$.
     
     \ref{item:PRsp-2}$\implies$\ref{item:PRsp-1}: This is obvious since zero-sets are $G_\delta$-sets (see, e.g., \cite[p.~15]{gillmanRingsContinuousFunctions1960}).

     \ref{item:PRsp-4}$\iff$\ref{item:PRsp-5}: Condition \ref{item:PRsp-4} is equivalent to each closed set being the closure of the union of zero-sets contained in it. By going to complements, this is equivalent to each open set being the interior of the intersection of the cozero-sets containing it.
    \end{proof}
Since meets in the frame $\fO X$ are calculated as interiors of  intersections, the last condition of \cref{PR for spaces} means that $\Coz X$ meet-generates $\fO X$. Also, $X$ being completely regular amounts to $\Coz X$ join-generating $\fO X$. This leads to our key definition: 

\begin{definition}\label{defn: PR}
    A frame $L$ is {\em perfectly regular} if $\Coz L$ join and meet generates $L$. 
\end{definition}

As an immediate consequence, we obtain:

\begin{theorem}\label{thm: spatial PR}
    For any topological space $X$,  $\fO X$ is a perfectly regular frame iff $X$ is a perfectly regular space.
\end{theorem}

Recall that a frame $L$ is {\em perfectly normal} provided $\Coz L = L$ (see, e.g., \cite[Cor.~VIII.1.2.1]{picadoSeparationPointfreeTopology2021}). Perfect regularity is stronger than complete regularity but weaker than perfect normality. An example of a completely regular frame that is not perfectly regular is given in \cref{rem: CR not PR}. 
It is based on the Lindel\"ofication of the frame of opens of the Michael line, the latter being our quintessential example of a perfectly regular frame that is not perfectly normal (see \cref{Michael is PR}). 


\begin{definition}
    The {\em Michael line} $\MM$ is the space whose points are those of $\RR$ and whose open sets  are of the form  
\[\setof{U\cup K}{U\mbox{ open in } \RR \mbox{ and } K\subseteq \PP},\] where $\PP$ is the set of irrational numbers. Let $\Ml$ denote the frame of opens of $\MM$.
\end{definition}
 
\begin{remark}\
\begin{enumerate}[label=\textup{(\arabic*)}]
    \item It is immediate from definition that the topology of $\MM$ is finer than that of $\RR$, that each irrational number is isolated, and that $\PP$ is open, so $\QQ$ is closed in $\MM$.
     \item 
In \cite[5.1.32]{engelkingGeneralToplogy1989}, the topology of the Michael line is denoted by $\RR_\QQ$. In  \cite[5C]{willardGeneralTopology1970}, the Michael line is called the {\em scattered line} and denoted by $\mathbf{S}$. Ma's blog \cite{maMichaelLineBasics2012} contains useful information about the Michael line.
\end{enumerate}
\end{remark}

\begin{theorem}    
\label{Michael is PR}
    $\Ml$ is perfectly regular but not perfectly normal.
\end{theorem}
\begin{proof} Since $\MM$ is not a perfectly normal space (see \cite[Ex.~1] {nyikosStructureZerodimensionalSpaces1975}), $\Ml$ is not a perfectly normal frame. Towards showing that $\Ml$ is perfectly regular, first observe that if $U$ is open in $\RR$ and $k \in \PP$, then $U\cup\set{k}$ is a cozero set of $\MM$ since $U$ is a cozero set of $\MM$ because it is a cozero set of $\RR$ ($\RR$ is a coarser topology and all opens are cozeros in $\RR$), and $\set{k}$ is clopen in $\MM$, so it too is cozero. Therefore, if 
    $K\subseteq \PP$ is countable then $U \cup K= \bigcup_{k\in K} (U\cup \set{k})$ is a countable union of cozero sets, hence is cozero.\footnote{Note that not all cozero sets of $\MM$ are of this form. For example, build a Cantor-type set $C$ within the irrationals. Then $C$ is clopen in $\MM$, hence cozero, but not of the form $U \cup K$ for countable $K$.}

    Next we show that $\Ml$ is completely regular. Let $U\cup K\in \Ml$. Then 
    $U\cup K=\bigcup_{k\in K} (U\cup \set{k})$ 
    and every $U\cup \set{k}$ is cozero by the previous paragraph. Thus, cozero elements join-generate $\Ml$.
    
    Finally, we show that the cozeros of $\Ml$ meet-generate $\Ml$. Let $U\cup K\in \Ml$. If $K$ is countable, then $U\cup K\in \Coz\Ml$. Thus, it suffices to show that $K$ is meet-generated by $\Coz \Ml$ when $K$ is uncountable. Clearly, 
    $K=\bigcap_{x\notin K} (\RR\setminus \set{x})$.
    This intersection is in fact a meet in $\Ml$ since $K$ is open in $\MM$. Moreover, every $\RR\setminus \set{x}$ is open in $\RR$, so $\RR\setminus \set{x}\in \Coz\fO\RR \subseteq \Coz\Ml$.
    Thus, cozero elements meet-generate $\Ml$. Consequently, $\Ml$ is perfectly regular.
\end{proof}

We conclude this section by  establishing our first characterization of perfectly regular frames. For this we need to remind the reader of Dedekind-MacNeille and Bruns-Lakser constructions. The {\em Dedekind-MacNeille completion} (aka normal completion) of a poset $A$, denoted $\dm (A)$, may be characterized as being the only, up to isomorphism, complete lattice in which $A$ is both join and meet-generating (see, e.g., \cite[Thm.~XII.2.7]{balbesDistributiveLattices1974}). 
The {\em Bruns-Lakser completion} of a meet-semilattice $A$, denoted $\bl (A)$, is the injective hull of $A$ in the category of meet-semilattices \cite{brunsInjectiveHullsSemilattices1970}. and it can be constructed as the frame of D-ideals of $A$. 
In the next characterization, we use that 
a frame $L$ is completely regular precisely when $\Coz L$ join-generates $L$ (hence our need to assume that $L$ is completely regular). From now on, the role of the meet-semilattice $A$ join-generating $L$ will always be played by $\Coz L$. 


\begin{theorem}\label{prop: char of PR}
    For a frame $L$, the following are equivalent:
    \begin{enumerate}[label=\upshape(\arabic*), ref=\arabic*]
        \item \label{item:PR-reg} $L$ is perfectly regular.
        \item \label{item:PR-dm}  $L \cong \dm(\Coz L)$.
        \item \label{item:PR-bl}  $L \cong \bl(\Coz L)$.
     \end{enumerate}
\end{theorem}
\begin{proof}
 (\ref{item:PR-reg})$\iff$(\ref{item:PR-dm}): By definition, $L$ is perfectly regular provided $\Coz L$ join and meet-generates $L$. Since $L$ is a complete lattice, this is equivalent to $L$ being isomorphic to $\dm(\Coz L)$.
 
 (\ref{item:PR-dm})$\iff$(\ref{item:PR-bl}): By \cite[Prop.~5]{banaschewskiSigmaFrames1980}, $\Coz L$ is a regular $\s$-frame. Thus, by \cite[Prop.~7.7]{ballLindelofTightnessDedekindMacNeille2017},
       $\bl(\Coz L) \cong \dm (\Coz L)$.
\end{proof}

\section{Perfect regularization}\label{sec: Perfect regularization}

We now use the Bruns-Lakser (equivalently, Dedekind-MacNeille) completion to associate with each completely regular frame a canonical perfectly regular frame. This procedure is akin to booleanization, and we term it {\em perfect regularization}. As we pointed out in the proof of \cref{prop: char of PR}, $\bl(\Coz L)\cong \dm(\Coz L)$. By \cite[Thm.~7.1, Prop.~7.7]{ballLindelofTightnessDedekindMacNeille2017}, both these frames are isomorphic to the smallest sublocale of $L$ containing $\Coz L$. We will denote this sublocale by $\pR L$.
The nucleus $\fp:L \to L$ associated with $\pR L$ is given, for each $a \in L$, by
\[
\fp(a) = \bigwedge\setof{c\in \Coz L}{a \leq c}.\]
The following is a quick consequence of \cref{prop: char of PR}.

\begin{theorem}\label{thm: L = pR L}
    A frame $L$ is perfectly regular iff $L = \pR L$.
\end{theorem}

\begin{proof}
    If $L = \pR L$, then $L$ is perfectly regular by \cref{prop: char of PR} since $\pR L \cong \bl(\Coz L) \cong \dm(\Coz L)$. For the converse, let $a \in L$. Since $L$ is perfectly regular, $a = \bigwedge \setof{c \in \Coz L}{a \le c}$, so $a = \fp (a)$, and hence $a \in \pR L$.
\end{proof}

Since $0_L \in \Coz L$, $\pR L$ is a dense sublocale of $L$. Thus, the {\em booleanization} 
\[\fB L := \setof{x^{**}}{x \in L}\] is always a sublocale of $\pR L$ since $\fB L$ is the smallest dense sublocale of $L$ (Isbell's density theorem; see, e.g., \cite[III.8.3]{picadoFramesLocalesTopology2012}).

\begin{proposition}\label{BL is pr}
    For any completely regular frame $L$, $\pR L$ is perfectly regular.
\end{proposition}

\begin{proof}
Set $M = \pR L$. 
    By \cref{prop: char of PR}, it is sufficient to show that 
    $M=\pR M$. The reverse inclusion is always true since $\pR M$ is always a sublocale of $M$.
    For the forward inclusion, recall that for any sublocale $S$ of $L$, we have $\Coz L \cap S \subseteq \Coz S$ (see, e.g., \cite[p.~354]{ballLindelofTightnessDedekindMacNeille2017}). Therefore, since 
    $\Coz L \subseteq \pR L=M$,  
    \[
    \Coz L = \Coz L \cap M \subseteq \Coz M \subseteq \pR M.
    \]
   Thus, both $\pR M$ and $M$ are sublocales of $L$ containing $\Coz L$. Since $M$ is the least sublocale of $L$ containing $\Coz L$, $M\subseteq \pR M$.
    \end{proof}

This result justifies the following:

\begin{definition}
    For a completely regular frame $L$, the sublocale $\pR L$ is called the \emph{perfect regularization} of $L$.
\end{definition}
 
For the next result, we recall that  regular Lindel\"of frames are a coreflective subcategory of the category of completely regular frames \cite{maddenLindelofLocalesRealcompactness1986}. By contrast, regular Lindel\"of spaces do not form a reflective subcategory of completely regular spaces.\footnote{This discrepancy can be explained by the fact that the frame of opens $\fO(\prod X_i)$ of the product of a family of topological spaces differs from the coproduct (see, e.g., $\coprod \fO X_i$ \cite[IV.5.4]{picadoFramesLocalesTopology2012}).}
Recalling that a $\sigma$-ideal of $\Coz L$ is an ideal closed under countable joins, the {\em Lindel\"of coreflection} $\lambda L$ of $L$ is given by the frame of $\sigma$-ideals of $\Coz L$, with the coreflection map given by join (see, e.g., \cite[Sec.~3]{maddenLindelofLocalesRealcompactness1986}). By associating $L$ with its image under the right adjoint of this map, $L$ may be thought of as a sublocale of $\lambda L$. 
Note that $\pR L$ is the smallest frame respecting the existing ``well behaved'' joins in $\Coz L$ \cite{brunsInjectiveHullsSemilattices1970}, whereas $\lambda L$ ``freely'' adds all joins \cite{banaschewskiFrameEnvelopeSframe1993}. Only in special cases will both constructions coincide, see \cref{BL Coz is Lind} below. By \cite[Thm.~5.5]{ballLindelofTightnessDedekindMacNeille2017}, the Lindel\"of coreflection has the same cozeros as a given frame (thought of as a sublocale of it). We thus obtain:
 
\begin{proposition}\label{Lind has same PR} A completely regular frame and its Lindel\"of coreflection have the same perfect regularization.
\end{proposition}

\begin{remark}
    The above is true for any $\lambda$-tight sublocale of a given locale. For details on $\lambda$-tightness, see \cite{ballLindelofTightnessDedekindMacNeille2017}, and in particular note that $\lambda$-tightness is equivalent to being densely C-embedded. 
\end{remark}

It is well known that $\MM$ is not Lindel\"of (see, e.g., \cite{maMichaelLineBasics2012}). Therefore, $\Ml$ is not  Lindel\"of. This together with \cref{Lind has same PR} yields:

\begin{proposition}\label{cor: Lind coreflection of OM}
    The Lindel\"of coreflection of $\Ml$ is not perfectly regular. 
\end{proposition}
\begin{proof}
    Since $\Ml$ is not Lindel\"of, $\lambda \Ml \ncong \Ml$. On the other hand, $\pR (\lambda \Ml) \cong \pR (\Ml)$ by \cref{Lind has same PR}, and $\pR (\Ml) \cong \Ml$ by \cref{Michael is PR}. Thus, $\lambda \Ml$ is not perfectly regular.
\end{proof}
 
\begin{remark}\label{rem: CR not PR}
As an immediate consequence of \cref{cor: Lind coreflection of OM}, we obtain that the frame witnessing that perfect regularity is strictly stronger than complete regularity is the Lindel\"of coreflection of $\Ml$. We thus have the strict implications
\[
\mbox{Perfectly regular } \implies \mbox{ completely regular } \implies\mbox{ regular},
\]
which mirror the implications involving normality (see, e.g., \cite{ferreiraCompletelyNormalFrames2009},\cite{avilezPointfreeStudyZembeddings2022}):
\[
  \mbox{Perfectly normal } \implies \mbox{ completely normal } \implies \mbox{  normal.}
\]
More examples of completely regular frames that are not perfectly regular will be given below (see \cref{P-space not PR,PN vs cozero vs Oz,Examples of PR and cozero}).
  \end{remark}

For a completely regular frame $L$, recall that a dense sublocale $S$ is {\em well-embedded} in $L$ (for a definition, see \cite{ballWellembedding$G_d$densityPointfree2006}) iff $\Coz L \subseteq \Coz S$ (see \cite[Lem.~5.3]{ballLindelofTightnessDedekindMacNeille2017}). As seen in the proof of \cref{BL is pr}, $\Coz L \subseteq \Coz (\pR L)$ and thus the perfect regularization is well-embedded in $L$. In fact, the perfect regularization is the smallest dense well-embedded sublocale of $L$ as is discussed after Lemma 5.3 in \cite{ballLindelofTightnessDedekindMacNeille2017}. This leads to one more characterization of perfect regularity.
 
\begin{theorem}\label{thm: another char of PR}
    A completely regular frame $L$ is perfectly regular iff $L$ has no proper dense well-embedded sublocales. 
\end{theorem}
\begin{proof}
    First let $L$ be perfectly regular. Since $\pR L$ is the smallest dense well-embedded sublocale, 
    if $S$ is dense and well-embedded then $\pR L \subseteq S$. But since $L$ is perfectly regular, $\pR L = L$ by \cref{prop: char of PR}, so $S=L$. For the converse, since $\pR L$ is dense and well-embedded, by hypothesis it is not proper, so $\pR L = L$. Hence, again by \cref{prop: char of PR}, $L$ is perfectly regular.
\end{proof}

For the next proposition, we use existing results about frame quotients, but interpret them in the language of sublocales. For any sublocale $S$ of $L$, recall that $S$ is \emph{coz-embedded} if the corresponding quotient map $L \to S$ restricts to an onto map $\Coz L \to \Coz S$, and that $S$ is \emph{C-embedded} if for every frame map $\fO \RR \to S$ there is a frame map $\fO \RR \to L$ making the following diagram commute:
\[\begin{tikzcd}
	L && S \\
	\\
	&& {\fO \RR }
	\arrow[two heads, from=1-1, to=1-3]
	\arrow[dashed, from=3-3, to=1-1]
	\arrow[from=3-3, to=1-3]
\end{tikzcd}\]
By putting \cite[Thm.~7.2.3]{ballQuotientsPointfreeTopology2002} and \cite[Thm.~5.5] {ballLindelofTightnessDedekindMacNeille2017} together, for a dense sublocale $S$ of a completely regular frame $L$, the following are equivalent:
\begin{itemize}
    \item $S$ is C-embedded.
    \item $S$ is well-embedded and coz-embedded.
    \item $\Coz S = \Coz L$.
\end{itemize}
 
\begin{proposition}\label{BL Coz is Lind}
    Let $L$ be completely regular. If the perfect regularization of $L$ is Lindel\"of, then $L$ is perfectly regular, and so $L$ is Lindel\"of.
\end{proposition}
\begin{proof} 
By \cite[Prop.~3.2]{dubeCozontoFrameMaps2006}, 
every Lindel\"of sublocale is coz-embedded. Therefore,  
    if $\pR L$ is Lindel\"of, then $\pR L$ is coz embedded. As mentioned before \cref{thm: another char of PR}, the perfect regularization is always well-embedded, and hence by the above, $\Coz L = \Coz (\pR L)$. Thus, $\lambda L = \lambda (\pR L)= \pR L$, where the first equality follows from the definition of $\lambda$ and the second because the perfect regularization is Lindel\"of by hypothesis. Since $L$ is always a sublocale sandwiched between its perfect regularization and its Lindel\"of coreflection, the result follows.  
\end{proof}
Any boolean frame is trivially perfectly regular, so in particular the booleanization $\fB L$ of a completely regular frame $L$ is such. However, $\pR L$ is in general different from the booleanization. For example, we show that the booleanization of the Michael line, $\fB \Ml$, is isomorphic to the powerset of $\PP$, and hence it is different from $ \pR(\fO M)\cong \fO M $. 

\begin{proposition}\label{booleanization of M}
    The booleanization of $\Ml$ is isomorphic to the powerset of the irrationals; that is, $\fB(\Ml) \cong \mathcal{P}(\PP)$.
\end{proposition}
\begin{proof}
    Recalling that each open of the Michael line is of the form $U \cup K$, where $U$ is open in $\RR$ and $K \subseteq\PP$, define the onto frame homomorphism $q\colon \Ml\to \mathcal{P} \PP$ by $q(U\cup K)=(U\cup K)\cap \PP$.
    If $q(U\cup K)=\varnothing$, then $(U\cap \PP)\cup K = \varnothing$, so $K=U=\varnothing$, and hence $U\cup K=\varnothing$.  
    Therefore, $\mathcal{P} \PP$ is a dense quotient of $\Ml$. Thus, $\mathcal{P} \PP$ is isomorphic to a dense boolean sublocale $S$ of $\Ml$. Since $S$ is dense, $\fB\Ml \subseteq S$ by Isbell's density theorem; and since $S$ is boolean, $S \subseteq \fB\Ml$ (because pseudocomplements in dense sublocales are calculated as in the ambient frame). Consequently, $S = \fB\Ml$, and hence $\mathcal{P} \PP\cong\fB\Ml$.
   \end{proof}

We next characterize when booleanization and perfect regularization coincide. 
Recall that a frame $L$ is \emph{almost P} if $\Coz L\subseteq \fB L$\footnote{Such frames are also known as \emph{$\omega_1$-hollow} (see \cite[6.1]{ballLhollowFramesLrepletions2024}).} (see \cite[Defn.~8.4.6]{ballQuotientsPointfreeTopology2002}). 
Also recall that a meet-semilattice $A$ is \emph{$\wedge$-subfit} if for all $a \nleq b$ in $A$, there exists $c \in A$ such that $c \meet a \neq 0 = c \meet b$.\footnote{These are also known as {\em separative} \cite[Def.~14.8]{jechSetTheory2003} or having {\em disjunction property} \cite[Lem.~3]{wallmanLatticesTopologicalSpaces1938}.}

\begin{theorem}\label{char of almost P}
    For a completely regular frame $L$, the following are equivalent.
    \begin{enumerate}[label=\textup{(\arabic*)}]
        \item\label{item:aP-1} $L$ is an almost $P$-frame.
         \item\label{item:aP-2} The booleanization 
        is well-embedded.
        \item\label{item:aP-3} The perfect regularization and booleanization coincide; that is, $\pR L = \fB L$.
        \item\label{item:aP-4} The perfect regularization is boolean.
        \item\label{item:aP-5} $\Coz L$ is $\wedge$-subfit.
       \end{enumerate}
\end{theorem}
\begin{proof}
\ref{item:aP-1}$\iff$\ref{item:aP-2}: See \cite[Prop. 6.5]{ballLindelofTightnessDedekindMacNeille2017}.
    
    \ref{item:aP-1}$\implies$\ref{item:aP-3}: Since $L$ is almost P, $\Coz L \subseteq \fB L$, so $\pR L \subseteq \fB L$ because it is the least sublocale containing $\Coz L$. The other inclusion always holds (see the beginning of the section).
       
       \ref{item:aP-3}$\implies$\ref{item:aP-4}: This is obvious.
       
     \ref{item:aP-4}$\implies$\ref{item:aP-1}: If $\pR L$ is boolean then since it is a dense sublocale of $L$, each element in it is equal to its own double pseudocomplement, and so $\pR L\subseteq \fB L$. In particular, $\Coz L \subseteq \fB L$, and hence $L$ is an almost P-frame.

    \ref{item:aP-4}$\iff$\ref{item:aP-5}: Since $\pR L \cong \bl (\Coz L)$, this follows from \cite[Thm.~6.3]{bezhanishviliDedekindMacNeilleRelatedCompletions2025}.
\end{proof}

\begin{remark}
    In the above proposition, condition \ref{item:aP-5} gives a characterization of almost P-frames that is ``internal'' to $\Coz L$. This characterization was proved indirectly in \cite[Prop.~5.2]{ballLhollowFramesLrepletions2024}. 
    It is worth giving a direct proof of this equivalence: 

\ref{item:aP-1}$\implies$\ref{item:aP-5}: Let $a \nleq b$ in $\Coz L$. Since $L$ is almost P, $\Coz L \subseteq \fB L$, so $a \nleq b$ in $\fB L$. Since the latter is a boolean frame, it is $\wedge$-subfit. So, there is $x \in \fB L$ with $a \meet x \neq 0 = b \meet x$. Since $L$ is completely regular, $x = \bigvee \{c \in \Coz L \mid c \leq x \}$, and hence one such $c$ has non-zero meet with $a$, but zero meet with $b$, showing $\Coz L$ is $\meet$-subfit.

\ref{item:aP-5}$\implies$\ref{item:aP-1}: Let $c\in \Coz L$. Since $L$ is completely regular and $c\leq c^{**}$, to prove that $c=c^{**}$, it suffices to show that whenever $y\prec c^{**}$ for $y\in \Coz L$, then $y\leq c$.\footnote{Recall that $\prec$ stands for the {\em rather below} relation given by $a \prec b$ provided $a^* \vee b =1$.} Suppose $y\nleq c$. By $\wedge$-subfitness, there is $x\in \Coz L$ such that $y\wedge x\neq 0 = c\wedge x$. We have
\[ c\meet x=0\implies x\leq c^*\implies c^{**}\leq x^*.\]
Since $y\prec c^{**}$, we obtain $y\prec x^*$, so 
$y \le x^*$. Thus, $y\meet x=0$, a contradiction.
\end{remark}

As an immediate consequence of \cref{char of almost P}, we obtain:

\begin{corollary}
     The frame of opens of the Michael line 
     is not an almost $P$-frame.
\end{corollary}
\begin{proof}
    Since $\Ml$ is perfectly regular, $\pR(\Ml) = \Ml$. But $\fB\Ml \neq \Ml$ by \cref{booleanization of M}. Thus, \cref{char of almost P}\ref{item:aP-3} fails for $\Ml$, and hence $\Ml$ is not almost P.
\end{proof}

Clearly, a completely regular frame $L$ is almost P ($\Coz L\subseteq \fB L$) and perfectly normal ($\Coz L = L$) iff $L$ is boolean\footnote{That is, a complete boolean algebra.}. The next result shows that perfect normality may be weakened to perfect regularity.     

\begin{theorem}\label{almost P + PR= boolean}
    A completely regular frame $L$ is boolean iff $L$ is almost P and perfectly regular.
\end{theorem}
\begin{proof}
    Combine \cref{thm: L = pR L,char of almost P}.
\end{proof}

\begin{remark}\label{P-space not PR}
    We now give another example of a completely regular frame that is not perfectly regular. Unlike the example in \cref{rem: CR not PR}, this one is spatial. Recall that $X$ is an almost P-space iff $\fO X$ is an almost P-frame. By \cref{almost P + PR= boolean}, any completely regular almost P-space that is not discrete is not perfectly regular. Consider, for example, the one-point compactification of an uncountable discrete space \cite[Example 2]{levyAlmostPspaces1977}. 
\end{remark}

 We next consider what it takes for a perfectly regular frame to be perfectly normal. We recall that a frame $L$ is a \emph{cozero frame} if $\Coz L$ is a frame. By  \cite[Rem.~7.2.1]{bezhanishviliSemilatticeBaseHierarchy2024}), we have the following characterization:
\begin{proposition}\label{prop: cozero frames} 
For a completely regular frame $L$, the following are equivalent.
    \begin{enumerate}[label=\textup{(\arabic*)}]
        \item $L$ is a cozero frame.
        \item $\Coz L$ is the perfect regularization; that is, $\Coz L = \pR L$.
        \item $\Coz L$ is a sublocale of $L$.
        \item $x \mapsto\bigwedge\setof{c \in \Coz L}{x\leq c}$ is a nucleus with fixed points $\Coz L$. 
    \end{enumerate}
\end{proposition}

We digress slightly to give the reader some understanding of where the notion of being a cozero frame fits into some other better studied classes of frames. Recall that a frame is \emph{Oz} if pseudocomplemented elements are cozeros, or equivalently, $\fB L \subseteq \Coz L$ \cite[Prop.~2.2]{banaschewskiOzPointfreeTopology2009}. 

\begin{proposition}\label{PN vs cozero vs Oz}
    For a completely regular frame $L$, consider the following three conditions.
    \begingroup
    \begin{enumerate}[label=\textup{(\arabic*)}]
        \item\label{item:PN-1} $L$ is perfectly normal.
        \item\label{item:PN-2} $L$ is a cozero frame.
        \item\label{item:PN-3} $L$ is  an $Oz$-frame.
    \end{enumerate}
    We have \ref{item:PN-1} $\implies$ \ref{item:PN-2} $\implies$ \ref{item:PN-3} and the reverse implications do not hold.
    \endgroup
\end{proposition}
\begin{proof}
    $\ref{item:PN-1}\implies\ref{item:PN-2}$ is obvious, and for
    $\ref{item:PN-2}\implies\ref{item:PN-3}$ see \cite[Thm.~7.10]{bezhanishviliSemilatticeBaseHierarchy2024}.
    Examples showing that the reverse implications do not hold are given in the same article:  
    7.14 provides an example of a completely regular extremally disconnected frame that is not a cozero frame (each extremally disconnected frame is obviously Oz); 
    whereas 7.7 gives an example of a completely regular frame that is cozero but not perfectly normal\footnote{Another example is given by any completely regular, non-boolean, almost boolean frame (see Remark~7.4 of the abovementioned article)}.
\end{proof}
 
 We now show that the missing condition for a perfectly regular frame to be a perfectly normal frame is exactly being a cozero frame. 

\begin{theorem}\label{PN=cozero+PR}
A completely regular frame $L$ is perfectly normal iff it is cozero and perfectly regular. 
\end{theorem}
\begin{proof}
We have
$\Coz L \subseteq \pR L \subseteq L$. The first containment is equality iff $L$ is a cozero frame (see \cref{prop: cozero frames}); the second containment is equality iff $L$ is perfectly regular (see \cref{thm: L = pR L}); and all are equal iff $L$ is perfectly normal, yielding the result. 
\end{proof}

\cref{PN=cozero+PR} allows us to see that being perfectly regular and cozero are independent notions. 

\begin{example} \label{Examples of PR and cozero}\
    \begin{enumerate}
        \item The Michael line is perfectly regular, but not perfectly normal (see \cref{Michael is PR}). Therefore, by \cref{PN=cozero+PR}, it is not cozero (that is, its cozeros do not form a frame).
        \item \label{item: examples-beta N} 
        For another example, let $L$ be the frame of opens of the Stone-\v Cech compactification ${\beta\NN}$ of the discrete space $\NN$. Then $L$ is extremally disconnected, so Oz, but not cozero (see \cite[Ex.~7.14]{bezhanishviliSemilatticeBaseHierarchy2024}). As we pointed out before \cref{thm: another char of PR}, $\pR L$ is a dense well-embedded sublocale of $L$. Since $L$ is Oz, it follows from \cite[Prop.~4.3.9]{avilezPointfreeStudyZembeddings2022} that $\pR L$ is coz-embedded, and hence $\Coz (\pR L) = \Coz L$. Since $\Coz L$ is not a frame, neither is $\Coz (\pR L)$, and thus $\pR L$ is not cozero. Thus, it cannot be perfectly normal by \cref{PN=cozero+PR}.
        \item The examples referred to in \cref{PN vs cozero vs Oz} which are cozero but not perfectly normal are not perfectly regular by \cref{PN=cozero+PR} (that is, the perfect regularizations are strict sublocales of their ambient frames).
        \item For another example, the Lindel\"of coreflection of any non-Lindel\"of boolean frame (for instance, the topology of an uncountable discrete space) is cozero and not perfectly normal, so not perfectly regular (again by \cref{PN=cozero+PR}).
    \end{enumerate}
\end{example}
\begin{remark}
    Let $L$ be the frame in \cref{Examples of PR and cozero}(\ref{item: examples-beta N}). 
    Since $L$ is Oz and $\pR L$ is a dense sublocale of $L$, 
    \[\fB(\pR L) =\fB L\subseteq \Coz L =\Coz (\pR L).\]
    Thus, $\pR L$ is Oz. Consequently, one cannot weaken the condition of being cozero to being Oz in \cref{PN=cozero+PR}. In other words, perfect regularity and Oz do not imply perfect normality.
\end{remark}

\section{A characterization of perfect regularity via sublocales}

In this section we give two more 
characterizations of perfect regularity, this time using the machinery of sublocales.  
For any frame $L$, let $\Ss L$ be the coframe of sublocales of $L$, which may be thought of as the lattice of generalized subspaces. Because of this, $\Ss L$ is a fundamental object of study in locale theory. An important difference between $\Ss L$ and the powerset of subspaces of a space is that $\Ss L$ is often only a coframe and not a complete boolean algebra. 
We refer to \cite[Ch.~III]{picadoFramesLocalesTopology2012} for various facts about sublocales that we use below.

For $a\in L$, we recall that the {\em open} and {\em closed} sublocales associated with $a$ are 
\[\open(a):= \setof{a \rightarrow x}{x \in L} = \setof{x\in L}{a \rightarrow x=x} \quad \text{and}\quad\close(a):=
\setof{x\join a}{x\in L} = {\uparrow}a.\]
  Here $\to$ denotes {\em relative pseudocomplement} (aka {\em Heyting implication}) in $L$  determined by 
  \[b\le a\to x \iff a\wedge b \le x\]
  for all $b \in L$. 
It is well known that $\open(a)$ and $\close(a)$ are complemented in $\Ss L$. 

Let $\close(L) := \{ \close(a) \mid a \in L \}$ and $\open(L) := \{ \open(a) \mid a \in L \}$ denote the closed and open sublocales of $L$,  respectively.
We freely use that $L\cong \open(L)\cong \close(L)^{op}$ 
(see, e.g., \cite[II.2.6]{johnstoneStoneSpaces1982} or \cite[III.6.1.5]{picadoFramesLocalesTopology2012}). In fact, the isomorphisms are obtained by $a \mapsto \open(a)$ and $a \mapsto \close(a)$.

The joins in $\close(L)$ are not the same as those in $\Ss(L)$ because arbitrary joins of closed sublocales are not closed. These can be described by

\[\bigsqcup_{i}\close(a_i)=\overline{\bigvee_{i}\close(a_i)}\] 
where $\bigsqcup$ is the join in $\close(L)$, $\bigvee$ is the join in $\Ss(L)$,  and closure denotes the smallest closed sublocale containing the join.
Similarly, \[ \bigsqcap \open(a_i)=\Int\left(\bigcap \open(a_i)\right)\]  where $\bigsqcap$ is the meet in the frame $\open (L)$ and interior denotes the largest open sublocale contained in the intersection.

We first give a characterization of when $\Coz L$ meet-generates $L$ without the assumption of complete regularity. 

\begin{proposition}\label{sublocale char of PR}
    For any frame $L$, the following are equivalent.
    \begin{enumerate}[label=\textup{(\arabic*)}]
        \item\label{item:subPR-1} $L$ is meet-generated by $\Coz L$.

        \item\label{item:subPR-2} $\close (a)=\bigvee\setof{\close(b)}{b\in \Coz L,\ a\leq b}$  for each $a\in L$.
         \item\label{item:subPR-3} $\close (a)=\bigsqcup\setof{\close(b)}{b\in \Coz L,\ a\leq b}$ for each $a\in L$.
         \item \label{item:subPR-4} $\open (a)=\bigsqcap \setof{\open(b)}{b\in \Coz L,\ a\leq b}$ for each $a\in L$.
         \end{enumerate}
\end{proposition}
\begin{proof}
    \ref{item:subPR-1}$\implies$\ref{item:subPR-2}: Let $a \in L$ and set $S = \bigvee\setof{\close(b)}{b\in \Coz L,\ a\leq b}$. We must show that $S=\close (a)$. From $a \le b$ it follows that $\close(b) \subseteq\close(a)$, so $S \subseteq \close(a)$. For the reverse inclusion, let $d \in \close(a)$, so $d \ge a$. By the formula for joins in $\Ss(L)$ (see, e.g., \cite[III.3.2]{picadoFramesLocalesTopology2012}), 
    we have
    \begin{equation}\label{join}
        S = \left\{\bigwedge A\mid A\subseteq \bigcup\setof{\close(b)}{b\in \Coz L,\ a\leq b}\right\}.\tag{$\ast$}
    \end{equation}
    By assumption,
        $d=\bigwedge\setof{c\in \Coz L}{d\leq c}$.
        For each such $c$, $a\leq c$, so 
    $c\in \bigcup\setof{\close(b)}{b\in \Coz L,\ a\leq b}$. Therefore, $\setof{c\in \Coz L}{d\leq c}\subseteq \bigcup\setof{\close(b)}{b\in \Coz L,\ a\leq b}$, and hence $\bigwedge\setof{c\in \Coz L}{d\leq c}\in S$ by \eqref{join}, yielding that $d \in S$. Thus, $\close (a) \subseteq S$, hence the equality.

    \ref{item:subPR-2}$\implies$\ref{item:subPR-3}: This is obvious.

    \ref{item:subPR-3}$\iff $\ref{item:subPR-1}:
        Since $a \mapsto \close(a)$ is an isomorphism of $ L$ and $\close(L)^{op}$,
        \begin{align*}
            a=\bigwedge \setof{b\in \Coz L}{ a\leq b } \iff \close (a) &= \bigsqcup \setof{\close (b)}{b\in \Coz L,\   a\leq b  }. 
        \end{align*}

\ref{item:subPR-4}$\iff $\ref{item:subPR-1}:
     Since $a \mapsto \open(a)$ is an isomorphism of $L$ and $\open (L)$,
        \begin{align*}
            a=\bigwedge \setof{b\in \Coz L}{a\leq b } \iff \open (a) &= \bigsqcap \setof{\open (b)}{ b\in \Coz L,\   a\leq b  }. 
            \qedhere
        \end{align*}
\end{proof}

If $L$ is completely regular, then 
$\Coz L$ is join-generating in $L$. Thus, as an immediate consequence of  \cref{sublocale char of PR}, we obtain:

\begin{corollary}\label{sublocale char of PR - opens}
    For a completely regular frame $L$, the following are equivalent.
    \begin{enumerate}[label=\textup{(\arabic*)}]
        \item\label{item:subPRop-1}
        $L$ is perfectly regular.
        \item\label{item:subPRop-2a} $\close (a)=\bigvee\setof{\close(b)}{b\in \Coz L,\ a\leq b}$  for each $a\in L$.
         \item\label{item:subPRop-3a} $\close (a)=\bigsqcup \setof{\close(b)}{b\in \Coz L,\ a\leq b}$ for each $a\in L$.
         \item\label{item:subPRop-3} $\open (a)=\bigsqcap\setof{\open(b)}{b\in \Coz L,\ a\leq b}$ for each $a\in L$.
    \end{enumerate}
\end{corollary}

\begin{remark}\label{rem: failure of intersection}    
 For perfectly normal frames, a strengthening of condition \ref{item:subPRop-3} holds trivially, namely:

 \[\open (a)= \bigcap\setof{\open(b)}{b\in \Coz L,\ a\leq b}\text{ for each }a\in L.\]
We show that this is no longer the case for perfectly regular frames. More precisely, we show that the frame of opens of the Michael line 
does not satisfy it. 
In fact, we show that
\[\open(\PP) \ne \bigcap\setof{\open(C) }{C\in \Coz \MM,\ \PP\subseteq C}.\]
First observe that $\PP \notin \open(\PP)$ since $\PP \to \PP = \MM \ne \PP$. On the other hand, we show that $\PP \in \open(C)$ for each $C\in \Coz \MM$ such that $\PP\subseteq C$. To see this, let $C$ be a cozero of $\MM$ containing $\PP$. Then $A: = \MM \setminus C$ is a zero set of $\MM$, so it is $G_\delta$ in $\MM$. Since $A \subseteq \QQ$, it follows from the definition of the topology of $\MM$ that $A$ is also a  $G_\delta$-set in $\RR$. 

We show that $A$ is scattered. (Recall that a space is {\em scattered} if each nonempty subspace has an isolated point.)
Use the Cantor-Bendixson Theorem (see, for instance, \cite[1.7.10]{engelkingGeneralToplogy1989}) to write $A$ as the disjoint union $A = D \cup S$, where $D$ is a dense-in-itself closed subspace of $A$ and $S$ is an open scattered subspace of $A$. It is sufficient to show that $D = \varnothing$. Suppose otherwise, then $D$ is a denumerable dense-in-itself subspace of $\RR$, so is homeomorphic to $\QQ$ (see, e.g., \cite[6.2.A(d)]{engelkingGeneralToplogy1989}). Since $A$ is $G_\delta$ in $\RR$, $A= \bigcap_{n=1}^\infty U_n$, where each $U_n$ is open in $\RR$. Now 
\[D = A \setminus S = \bigcap_{n=1}^\infty U_n \cap \bigcap_{s \in S} (\RR\setminus\{s\}),\]
and because $S$ is countable, $D$ is a $G_\delta$ in $\RR$, contradicting that $\QQ$ is not $G_\delta$ in $\RR$ (the latter being a consequence of the Baire Category theorem; see, e.g., \cite[3.9.B]{engelkingGeneralToplogy1989}). Thus, $D = \varnothing$, so $A=S$, and hence $A$ is scattered.  

Now, since $A$ is scattered, $\Int(A \cup \PP) = \PP$:
if $x \in \Int(A \cup \PP)$ and $x \notin \PP$ then there is an open interval $I_x$ such that $I_x \subseteq A \cup \PP$, so $I_x \setminus \PP \subseteq A$. Because $I_x \setminus \PP$ is dense-in-itself and $A$ is scattered, $I_x$ must be empty. Thus, $x\in \PP$.
Consequently, since the relative pseudocomplement in the frame of opens of a topological space $X$ is calculated by $$U \to V = \Int((X\setminus U)\cup V),$$ we obtain that $C \to \PP = \Int(A \cup \PP) = \PP$, which shows $\PP\in \open (C)$.
\end{remark}

\begin{remark}\label{Bool of michael is open}
    By \cref{booleanization of M}, $\fB(\Ml) \cong \mathcal{P}\PP$. The latter is the frame ${\downarrow}\PP = \{ U \in \Ml \mid U \subseteq \PP \}$ whose corresponding sublocale is $\open (\PP)$ (see, e.g., \cite[III.6.1.1.]{picadoFramesLocalesTopology2012}). Thus, the booleanization of $\Ml$ is the open sublocale $\open (\PP)$\footnote{This in itself is interesting as the booleanization is seldom an open sublocale.}, and hence $\fB(\Ml) \ne \bigcap\setof{\open(C) }{C\in \Coz \MM,\ \PP\subseteq C}$ by \cref{rem: failure of intersection}. 
\end{remark}

The intersection defined in the remark above is clearly a closure type operation and thus leads to the following.
\begin{definition}
  For any sublocale $S$ of $L$, the  
 \emph{coz-closure}\footnote{Dube in \cite[Rem.~3.3]{dubeNoteWeaklyPseudocompact2017} calls it the $\delta$-closure -- our new name seems more fitting as is made clear in \cref{rem: Mrowka}.} of $S$ is defined to be the intersection of all cozero sublocales containing $S$; that is, 
\[\overline{S}^{\,{coz}} = \bigcap \setof{\open(a)}{a \in \Coz L,\, S \subseteq \open(a)}.\]
\end{definition}

As an immediate consequence of \cref{sublocale char of PR - opens}, we obtain our final characterization of perfect regularity.
\begin{corollary} 
A completely regular frame $L$ is perfectly regular iff every open sublocale is the interior of its $coz$-closure; that is, $\open (a) = \Int \left(\overline {\open (a)}^{\,coz}\right)$ for each $a \in L$.
\end{corollary}

\begin{remark}\label{rem: Mrowka}
    In spaces, coz-closure is closely related to Mrowka's notion of $Q$-closure \cite{mrowkaPropertiesQspaces1957} (aka realclosure \cite[Defn.~5.13]{aloNormalTopologicalSpaces1974} or $G_\delta$-closure \cite[p.~44]{blairExtensionsZerosetsRealvalued1974}): for a set $A$ in a space $X$, define the {\em $Q$-closure} of $A$ as the set of points $x \in X$ such that for any $G_\delta$-set $G$ containing $x$ we have $A\cap G\neq \varnothing$. It is not hard to see that the $Q$-closure coincides with 
\[\bigcap\setof{F}{A\subseteq F, F \text{ is an}\ F_\sigma\text{-set}}.\]
In a completely regular (Hausdroff) space $X$ the $Q$-closure may be described using cozeros; that is, 
\[\bigcap\setof{C}{A\subseteq C, C \in \Coz X}\]
(see comment before \cite[Cor.~2.6]{blairExtensionsZerosetsRealvalued1974}). This is the spatial version of the coz-closure defined above.

In a locale $L$, Mrowka's closure would translate to: 
\[\overline{S}^{\,{\delta}} = \bigcap\setof{F \in \Ss L}{ S \subseteq F, \, F\text{ is an }F_\sigma\text{-sublocale}    },\]
where a sublocale is $F_\sigma$ if it is a countable join of closed sublocales (see, e.g, \cite[p.~103]{gutierrezgarciaPointfreeFormsDowkers2009}). Recall that $\Ss L$ is a coframe but not necessarily a complete boolean algebra, so the relationship between the various subsets is more complicated. 
In view of this, it of interest to investigate whether in a completely regular locale the coz-closure coincides with the one defined above. 
\end{remark}

\section{Conclusions}



It is a celebrated result in locale theory, known as Isbell's density theorem, that each locale has a smallest dense sublocale, which is always boolean (the so-called booleanization). It provides an important motivation to study sublocales (rather than only subspaces) of a topological space since the booleanization is rarely spatial. In the setting of complete regularity, cozeros play a pivotal role, but this information 
is lost when moving to the booleanization (unless we are in the special case of almost P-locales).
The perfect regularization introduced in this paper remedies this by studying another dense sublocale, which retains all the information provided by the cozeros. We find it remarkable that this sublocale can be described using the Dedekind-MacNeille completion of the cozeros, which in turn is isomorphic to their injective hull. These equivalent ways of thinking about perfect regularization have been crucial in obtaining our main results. 

We view this paper as the launching pad for the study of perfect regularity, a fruitful yet overlooked higher separation axiom for both spaces and locales. Our considerations give rise to numerous questions and directions for further study, some of which we detail below.

\begin{enumerate}
    \item Like booleanization, perfect regularization is not always functorial. It is of interest to determine the conditions under which it becomes a functor.
    \item Again like booleanization, perfect regularization is not always spatial. It is natural to seek necessary and sufficient conditions for perfect regularization of a spatial frame to be spatial.
    \item While we defined perfect regularization for frames, perfect regularization for completely regular spaces still needs to be developed. 
    
    \item As we saw in \cref{rem: failure of intersection}, the condition that $\open (a)=\bigcap\setof{\open(b)}{ b\in \Coz L,\ a\leq b} $ for each $a\in L$ is an immediate consequence of perfect normality, but is strictly stronger than perfect regularity. It is of interest to determine its exact relationship to these two notions.     

   \item For a completely regular frame $L$, let $\fH L = \bigcap\setof{\open(b)}{b \mbox{ is a dense cozero}}$ be the ``cozero-hollowing'' of $L$. As was demonstrated in \cite{ballLhollowFramesLrepletions2024}, $\fH L$ plays an important role in the scaffolding of a completely regular frame. 
     The argument in \cref{Bool of michael is open} shows that 
     for perfectly regular frames, the cozero-hollowing does not always coincide with 
     the booleanization. 
     It is natural to study the exact relationship between $\pR L$ and $\fH L$.
   
\item It would be interesting to explore the idea of ``perfect regularity'' with respect to other bases, using the machinery of S-bases of \cite{bezhanishviliSemilatticeBaseHierarchy2024}.

 \item The notion of perfect normality was refined to that of $\kappa$-perfect normality in \cite[p.~25]{ballLhollowFramesLrepletions2024}. It would be natural to seek the same $\kappa$-refinement for our
 notion of perfect regularity.  
 
\end{enumerate}

\section*{Acknowledgments}

We are grateful to Alan Dow, K.~P.~Hart, and Jan van Mill for useful discussions.

\end{document}